\documentclass[11pt]{article}
\usepackage{amsmath,amsthm,amssymb,amsfonts,amscd}
\usepackage{amsxtra}
\usepackage{eucal}
\usepackage{mathrsfs}
\usepackage{color}
\usepackage{xcolor}
\usepackage{hyperref}
\usepackage{enumerate}
\usepackage{blkarray}
\usepackage{tikz}
\usepackage{enumitem}
\usepackage{tcolorbox}
\usepackage{algorithm}
\usepackage{algpseudocode}
\usepackage[a4paper,margin=2cm]{geometry}
\definecolor{forestgreen(traditional)}{rgb}{0.0, 0.27, 0.13}
\definecolor{forestgreen(web)}{rgb}{0.13, 0.55, 0.13}
\definecolor{airforceblue}{rgb}{0.36, 0.54, 0.66}

\newcommand\GF{\operatorname{GF}}

\newcommand\abs[1]{\lvert #1\rvert}

\newcommand{\rnk}{\mathrm{rank}}

\hypersetup{
  colorlinks,
  citecolor=forestgreen(web),
  linkcolor=airforceblue,
  urlcolor=magenta,
  pdftitle={Intertwining connectivities for vertex-minors and pivot-minors},
  pdfauthor={Duksang Lee and Sang-il Oum}}

\newtheorem{theorem}{Theorem}[section]
\newtheorem{lemma}[theorem]{Lemma}
\newtheorem{conjecture}[theorem]{Conjecture}

\newtheorem{corollary}[theorem]{Corollary}
\newtheorem{proposition}[theorem]{Proposition}

\theoremstyle{definition}

\theoremstyle{remark}

\usepackage{authblk}
\usepackage{lineno}
\title{Intertwining connectivities for vertex-minors and pivot-minors}
\author[2,1]{Duksang Lee\thanks{Supported by the Institute for Basic Science (IBS-R029-C1).}}
\author[$*$1,2]{Sang-il Oum}
\affil[1]{Discrete Mathematics Group,
Institute for Basic Science (IBS),
Daejeon,~South~Korea}
\affil[2]{Department of Mathematical Sciences, KAIST, Daejeon, South~Korea}
\affil[ ]{Email: \texttt{duksang@kaist.ac.kr}, \texttt{sangil@ibs.re.kr}}
\date{\today}	

\begin{document}
\maketitle
\begin{abstract}
We show that for pairs $(Q,R)$ and $(S,T)$ of disjoint subsets of vertices of a graph $G$, if $G$ is sufficiently large, then there exists a vertex $v$ in $V(G)-(Q\cup R\cup S\cup T)$ such that there are two ways to reduce $G$ by a vertex-minor operation that removes $v$ while preserving the connectivity between $Q$ and $R$ and the connectivity between $S$ and $T$.
Our theorem implies an analogous theorem of Chen and Whittle (2014) for matroids restricted to binary matroids.
\end{abstract}

\section{Introduction}
Oum~\cite{Oum2005} proved a vertex-minor analog of Tutte's Linking Theorem on matroids~\cite{Tutte1965}. Roughly speaking, the theorem of Oum says that for every pair of disjoint sets $Q$, $R$ of vertices of a graph $G$, there are at least two ways to reduce $G$ by a vertex-minor operation while keeping the `connectivity' between $Q$ and~$R$, where this connectivity will be defined using the rank function of matrices.
We prove that if the graph is large, for any two pairs $(Q,R)$ and $(S,T)$ of disjoint sets of vertices, there exist two ways to reduce the graph by a vertex-minor operation while preserving the connectivity between $Q$ and $R$, and the connectivity between $S$ and $T$.

To state the main theorem precisely,
we introduce a few concepts. A graph is \emph{simple} if it has neither loops nor parallel edges. In this paper, all graphs are finite and simple. For a vertex $v$ of a graph~$G$, the \emph{local complementation} at $v$ is an operation that, for each pair $x$, $y$ of distinct neighbors of $v$, adds an edge $xy$ if $x$ and $y$ are non-adjacent in $G$ and removes an edge $xy$ otherwise. Let $G*v$ be the graph obtained from $G$ by applying the local complementation at $v$. A graph $H$ is a \emph{vertex-minor} of $G$ if it can be obtained from $G$ by applying a sequence of local complementations and deletions of vertices. For an edge $uv$ of a graph $G$, let $G\wedge uv=G*u*v*u$. We remark that the pivoting operation is well defined since $G*u*v*u=G*v*u*v$. The operation obtaining $G\wedge uv$ from $G$ is called \emph{pivoting} $uv$.  A graph $H$ is a \emph{pivot-minor} of $G$ if it can be obtained from $G$ by applying a sequence of pivoting edges and deleting vertices.

For a graph $G$, the \emph{cut-rank} function $\rho_{G}$ is a function that maps a set $X$ of vertices of $G$ to the rank of an $X\times (V(G)-X)$ matrix\footnote{
  For two sets $A$ and $B$, an $A\times B$-matrix denotes an $\abs{A}\times \abs{B}$ matrix whose rows and columns are indexed by the elements of $A$ and $B$ respectively. 
}
over $\GF(2)$ whose $(i,j)$-entry is $1$ if $i$ and $j$ are adjacent and $0$ otherwise.
For disjoint sets $S$, $T$ of vertices of $G$, the \emph{connectivity between $S$ and $T$}, denoted by $\kappa_{G}(S,T)$, is defined by \[\min_{S\subseteq X\subseteq V(G)-T}\rho_{G}(X).\] 

Now we are ready to state the analog of Tutte's Linking Theorem for vertex-minors as reformulated by Geelen, Kwon, McCarty, and Wollan~{\cite[Theorem 4.1]{Geelen2020}}.
\begin{theorem}[Oum~\cite{Oum2005}]
\label{thm:oum_linking}
Let $G$ be a graph and $Q$, $R$ be disjoint subsets of $V(G)$. Let $\kappa_{G}(Q,R)=k$ and $F=V(G)-(Q\cup R)$. For each vertex $v$ of $F$, at least two of the following hold:
\begin{enumerate}[label=\rm(\roman*)]
\item $\kappa_{G\setminus v}(Q,R)=k$.
\item $\kappa_{G*v\setminus v}(Q,R)=k$.
\item $\kappa_{G\wedge uv\setminus v}(Q,R)=k$ for each neighbor $u$ of $v$.
\end{enumerate}
\end{theorem}
Theorem~\ref{thm:oum_linking} is about preserving the rank-connectivity of one pair of vertex sets while taking vertex-minors. Here is our main theorem which considers two pairs of vertex sets.

\begin{theorem}
\label{thm:main}
Let $G$ be a graph and $Q$, $R$, $S$, and $T$ be subsets of $V(G)$ such that $Q\cap R=S\cap T=\emptyset$. Let $\kappa_{G}(Q,R)=k$, $\kappa_{G}(S,T)=\ell$, and $F=V(G)-(Q\cup R\cup S\cup T)$. If $|F|\geq (2\ell+1)2^{2k}$, then there exists a vertex $v$ in $F$ such that at least two of the following hold:
\begin{enumerate}[label=\rm(\roman*)]
\item\label{item:m1} $\kappa_{G\setminus v}(Q,R)=k$ and $\kappa_{G\setminus v}(S,T)=\ell$.
\item\label{item:m2}  $\kappa_{G*v\setminus v}(Q,R)=k$ and $\kappa_{G*v\setminus v}(S,T)=\ell$.
\item\label{item:m3}  $\kappa_{G\wedge uv\setminus v}(Q,R)=k$ and $\kappa_{G\wedge uv\setminus v}(S,T)=\ell$ for each neighbor $u$ of $v$.
\end{enumerate}
\end{theorem}

Since at least two of \ref{item:m1}, \ref{item:m2}, and \ref{item:m3} hold, we deduce that \ref{item:m1} or \ref{item:m3} holds. Thus, we have the following corollary for pivot-minors.

\begin{corollary}
\label{cor:main}
Let $G$ be a graph and $Q$, $R$, $S$, and $T$ be subsets of $V(G)$ such that $Q\cap R=S\cap T=\emptyset$. Let $\kappa_{G}(Q,R)=k$, $\kappa_{G}(S,T)=\ell$, and $F=V(G)-(Q\cup R\cup S\cup T)$. If $|F|\geq (2\ell+1)2^{2k}$, then there exists a vertex $v$ in $F$ such that at least one of the following holds:
\begin{enumerate}[label=\rm(\roman*)]
\item\label{item:cormain1} $\kappa_{G\setminus v}(Q,R)=k$ and $\kappa_{G\setminus v}(S,T)=\ell$.
\item\label{item:cormain2} $\kappa_{G\wedge uv\setminus v}(Q,R)=k$ and $\kappa_{G\wedge uv\setminus v}(S,T)=\ell$ for each neighbor $u$ of $v$.
\end{enumerate}
\end{corollary}

Our proof is inspired by the proof of the following theorem of Chen and Whittle~\cite{Chen2014} who proved the analog for matroids, which was conjectured by Geelen, and proved for representable matroids by Huynh and van Zwam~\cite{Tony2014}.

\begin{theorem}[Chen and Whittle~\cite{Chen2014}]\label{thm:matroid}
Let $M$ be a matroid and $Q$, $R$, $S$, and $T$ be subsets of $E(M)$ such that $Q\cap R=S\cap T=\emptyset$. Let $\kappa_{G}(Q,R)=k$, $\kappa_{G}(S,T)=\ell$, and $F=E(M)-(Q\cup R\cup S\cup T)$. If $|F|\geq(2\ell+1)2^{2k+1}$, then there exists an element $e$ of $E(M)$ such that at least one of the following holds:
\begin{enumerate}[label=\rm(\roman*)]
\item\label{item:matroid1} $\kappa_{M\setminus e}(Q,R)=k$ and $\kappa_{M\setminus e}(S,T)=\ell$.
\item\label{item:matroid2} $\kappa_{M/e}(Q,R)=k$ and $\kappa_{M/e}(S,T)=\ell$.
\end{enumerate}
\end{theorem}

In fact, Corollary~\ref{cor:main} implies Theorem~\ref{thm:matroid} restricted to binary matroids by using a relation between pivot-minors of bipartite graphs and minors of matroids~\cite{Oum2005}.
One of the key differences between our proof and the proof of Chen and Whittle is that we use a new way of measuring the local connectivity, $\tilde\sqcap(S,T)=\frac{1}{2}(\rho_{G}(S)+\rho_{G}(T)-\rho_{G}(S\cup T))$. The purpose of having $\frac{1}{2}$ in the previous definition is to ensure that $\tilde\sqcap_{G}[S,V(G)-S]=\rho_{G}(S)$. 

Our theorem is motivated by the following conjecture for pivot-minors. A pivot-minor $H$ of a graph $G$ is \emph{proper} if $|V(H)|<|V(G)|$. A graph $G$ is an \emph{intertwine} of graphs $H_{1}$ and $H_{2}$ for pivot-minors if it contains both $H_{1}$ and $H_{2}$ as pivot-minors and no proper pivot-minor of $G$ contains both $H_{1}$ and $H_{2}$ as pivot-minors.

\begin{conjecture}[Intertwining conjecture for pivot-minors]
\label{conj:inter}
For graphs $G_{1}$ and $G_{2}$, there are only finitely many intertwines of $G_{1}$ and $G_{2}$ for pivot-minors.
\end{conjecture}

Together with Theorem~\ref{thm:oum_linking}, Conjecture~\ref{conj:inter} implies Corollary~\ref{cor:main} without an explicit function. 
Suppose that $G$ is a graph and $Q$, $R$, $S$, and $T$ are subsets of $V(G)$ such that $Q\cap R=S\cap T=\emptyset$, $\kappa_{G}(Q,R)=k$, and $\kappa_{G}(S,T)=\ell$.
By Theorem~\ref{thm:oum_linking}, $G$ has pivot-minors $G_{1}$ and $G_{2}$ such that $V(G_{1})=Q\cup R$, $V(G_{2})=S\cup T$, $\rho_{G_{1}}(Q)=k$, and $\rho_{G_{2}}(S)=\ell$. If Conjecture~\ref{conj:inter} holds, then there exists an integer $n$ such that every intertwine of $G_{1}$ and $G_{2}$ for pivot-minors has at most $n$ vertices. If $|V(G)|>n$, then $G$ is not an intertwine of $G_{1}$ and $G_{2}$ for pivot-minors. Hence, there exists a proper pivot-minor $H$ of $G$ having both $G_{1}$ and $G_{2}$ as pivot-minors. Let $v$ be a vertex in $V(G)-V(H)$. Then it is easy to see that ~\ref{item:cormain1} or~\ref{item:cormain2} of Corollary~\ref{cor:main} holds.

The following conjecture of Oum~\cite{Oum2016} implies the intertwining conjecture for pivot-minors.

\begin{conjecture}[Well-quasi-ordering conjecture for pivot-minors]
\label{conj:wqo}
For every infinite sequence $G_{1}$, $G_{2}$, $\ldots$ of graphs, there exist $i<j$ such that $G_{i}$ is isomorphic to a pivot-minor of $G_{j}$.
\end{conjecture}

Although the analog of Conjecture~\ref{conj:wqo} for vertex-minors is still open, Geelen and Oum~\cite{Oum2009} proved the analog of Conjecture~\ref{conj:inter} for vertex-minors.

This paper is organized as follows. In Section~\ref{sec:pre}, we introduce concepts of vertex-minors and pivot-minors, and review several inequalities for cut-rank functions. In Section~\ref{sec:pf}, we present simple lemmas on the cut-rank function. In Section~\ref{sec:mainpf}, we prove Theorem~\ref{thm:main}.

\section{Preliminaries}
\label{sec:pre}
For a graph $G$ and a vertex $v$ of $G$, let $N_{G}(v)$ be the set of vertices adjacent to $v$ in $G$.
For a graph~$G$ and a subset $X$ of $V(G)$, let $G[X]$ be the induced subgraph of $G$ on $X$.
For two sets $A$ and~$B$, let $A\triangle B=(A-B)\cup(B-A)$.

\paragraph{Vertex-minors and pivot-minors}
Note that for a graph $G$ and a vertex $v$ of $G$, the local complementation at $v$ replaces $G[N_{G}(v)]$ with its complement. A graph $H$ is \emph{locally equivalent} to a graph $G$ if $H$ can be obtained from~$G$ by applying a sequence of local complementations. 
Recall that a graph $H$ is a \emph{vertex-minor} of a graph $G$ if $H$ can be obtained from $G$ by applying a sequence of local complementations and deletions of vertices. 

For an edge $uv$ of a graph $G$, let $G\wedge uv=G*u*v*u$. Then $G\wedge uv$ is obtained from $G$ by \emph{pivoting} $uv$. 
Alternatively, pivoting $uv$ can be understood as an operation that removes an edge $xy$ if $x$, $y$ are non-adjacent 
and adds an edge $xy$ otherwise
for every pair $(x,y)\in (X_1\times X_2)\cup (X_2\times X_3)\cup (X_3\times X_1)$
where $X_1$ is the set of common neighbors of $u$ and $v$, 
$X_2$ is the set of neighbors of $u$ that are non-neighbors of $v$,
and $X_3$ is the set of neighbors of $v$ that are non-neighbors of $u$
and then swaps the labels of $u$ and $v$, see Oum~\cite{Oum2005} and Figure~\ref{fig:pivot}.
The graph $G\wedge uv$ is well defined since $G*u*v*u=G*v*u*v$ ~{\cite[Corollary 2.2]{Oum2005}}. 
A graph~$H$ is a \emph{pivot-minor} of a graph $G$ if $H$ can be obtained from $G$ by a sequence of pivoting and deleting vertices.

\begin{figure}
	\centering
	\begin{tikzpicture}
    \tikzstyle{v}=[circle, draw, solid, fill=black, inner sep=0pt, minimum width=3pt]
		\draw (-1,0) node[v,label=above:$u$](u){};
		\draw (1,0) node[v,label=above:$v$](v){};
		\draw (u)--(v);
		\draw (u)--(-0.2,-0.7);
		\draw (u)--(-0.3,-0.7);
		\draw (u)--(-0.4,-0.7);
		\draw (u)--(-1.2,-1.7);
		\draw (u)--(-1.3,-1.7);
		\draw (u)--(-1.4,-1.7);
		\draw (v)--(0.2,-0.7);
		\draw (v)--(0.3,-0.7);
		\draw (v)--(0.4,-0.7);
		\draw (v)--(1.2,-1.7);
		\draw (v)--(1.3,-1.7);
		\draw (v)--(1.4,-1.7);
		\draw [dashed, blue] (-0.4,-1.2)--(-1.2,-1.7);
		\draw [dashed, blue] (0,-1.2)--(-1.5,-2.6);
		\draw [dashed, blue] (0.4,-1.2)--(1.2,-1.7);
		\draw [dashed, blue] (0,-1.2)--(1.5,-2.6);
		\draw [dashed, blue] (-1.2,-1.7)--(1.2,-1.7);
		\draw [dashed, blue] (-1.5,-2.2)--(1.5,-2.2);
		\draw[draw, solid, black, fill=white] (0,-1) ellipse (0.6 and 0.4);
		\draw[draw, solid, black, fill=white] (-1.4,-2) ellipse (0.4 and 0.6);
		\draw[draw, solid, black, fill=white] (1.4,-2) ellipse (0.4 and 0.6);
	\end{tikzpicture}
	\qquad
	\begin{tikzpicture}
    \tikzstyle{v}=[circle, draw, solid, fill=black, inner sep=0pt, minimum width=3pt]
		\draw (-1,0) node[v,label=above:$v$](u){};
		\draw (1,0) node[v,label=above:$u$](v){};
		\draw (u)--(v);
		\draw (u)--(-0.2,-0.7);
		\draw (u)--(-0.3,-0.7);
		\draw (u)--(-0.4,-0.7);
		\draw (u)--(-1.2,-1.7);
		\draw (u)--(-1.3,-1.7);
		\draw (u)--(-1.4,-1.7);
		\draw (v)--(0.2,-0.7);
		\draw (v)--(0.3,-0.7);
		\draw (v)--(0.4,-0.7);
		\draw (v)--(1.2,-1.7);
		\draw (v)--(1.3,-1.7);
		\draw (v)--(1.4,-1.7);
		\draw [dashed, blue] (-0.4,-1.2)--(-1.5,-2.6);
		\draw [dashed, blue] (0,-1.2)--(-1.2,-1.7);
		\draw [dashed, blue] (0.4,-1.2)--(1.5,-2.6);
		\draw [dashed, blue] (0,-1.2)--(1.2,-1.7);
		\draw [dashed, blue] (-1.2,-1.7)--(1.5,-2.2);
		\draw [dashed, blue] (-1.5,-2.2)--(1.2,-1.7);
		\draw[draw, solid, black, fill=white] (0,-1) ellipse (0.6 and 0.4);
		\draw[draw, solid, black, fill=white] (-1.4,-2) ellipse (0.4 and 0.6);
		\draw[draw, solid, black, fill=white] (1.4,-2) ellipse (0.4 and 0.6);
	\end{tikzpicture}
	\caption{$G$ and $G\wedge uv$.}\label{fig:pivot}
\end{figure}
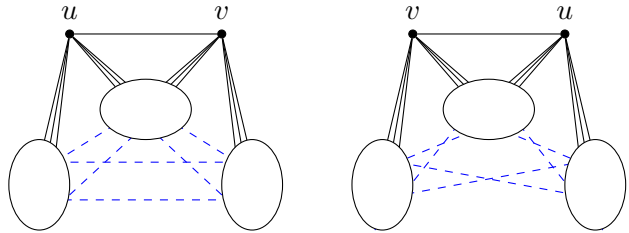

\begin{lemma}[Oum~\cite{Oum2005}]
\label{lem:local_equiv}
Let $G$ be a graph and $v$ be a vertex of $G$. If $x$ and $y$ are neighbors of $v$ in $G$, then $(G\wedge vx)\setminus v$ is locally equivalent to $(G\wedge vy)\setminus v$.
\end{lemma}

For a vertex $v$ of $G$ with a neighbor $u$, we write $G/v$ to denote $G\wedge uv\setminus v$. If $v$ has no neighbor in $G$, then we let $G/v$ denote $G\setminus v$. Then the graph $G/v$ is well-defined up to local equivalence by Lemma~\ref{lem:local_equiv}.
The following lemma can be easily deduced from isotropic systems~\cite{Bouchet1988}, and Geelen and Oum provide an elementary graph-theoretic proof.

\begin{lemma}[Geelen and Oum~{\cite[Lemma 3.1]{Oum2009}}]
\label{lem:onetoone}
Let $G$ be a graph and $v$ and $w$ be vertices of $G$. Then the following hold:
\begin{enumerate}[label=\rm(\arabic*)]
\item If $v\neq w$ and $vw\notin E(G)$, then $(G*w)\setminus v$, $(G*w*v)\setminus v$, and $(G*w)/v$ are locally equivalent to $G\setminus v$, $G*v\setminus v$, and $G/v$ respectively.
\item If $v\neq w$ and $vw\in E(G)$, then $(G*w)\setminus v$, $(G*w*v)\setminus v$, and $(G*w)/v$ are locally equivalent to $G\setminus v$, $G/v$, and $(G*v)\setminus v$ respectively.
\item If $v=w$, then $(G*w)\setminus v$, $(G*w*v)\setminus v$, and $(G*w)/v$ are locally equivalent to $G*v\setminus v$, $G\setminus v$, and $G/v$ respectively.
\end{enumerate}
\end{lemma}

From Lemma~\ref{lem:onetoone}, we can deduce the following lemma easily.

\begin{lemma}
\label{lem:perm}
Let $H$ be a vertex-minor of a graph $G$ and $v$ be a vertex of $H$. Let $H_{1}=H\setminus v$, $H_{2}=H*v\setminus v$, and $H_{3}=H/v$ and let $G_{1}=G\setminus v$, $G_{2}=G*v\setminus v$, and $G_{3}=G/v$. Then there exists a permutation $\sigma:\{1,2,3\}\rightarrow\{1,2,3\}$ such that $H_{i}$ is a vertex-minor of $G_{\sigma(i)}$ for each $i\in\{1,2,3\}$.
\end{lemma}

\begin{proof}

Since $H$ is a vertex-minor of $G$, there exist a sequence $u_{1},\ldots, u_{m}$ of vertices of $G$ and a subset $X$ of $V(G)$ such that $H=G*u_{1}*\cdots *u_{m}\setminus X$. We proceed by induction on $m$. If $m=0$, then $H=G\setminus X$. Obviously, $H_{i}=G_{i}\setminus X$ for each $i\in\{1,2\}$. We claim that $H_{3}=G_{3}\setminus X$. If there is a neighbor $w$ of $v$ in $G$ which is not in $X$, then $H_{3}=H\wedge vw\setminus v=(G\wedge vw\setminus v)\setminus X=G_{3}\setminus X$. If $N_{G}(v)\subseteq X$, then $H_{3}=H\setminus v=G\setminus X\setminus v$. Since $X$ contains all the neighbors of $v$, it is easy to check that $G_{3}\setminus X=((G\wedge uv)\setminus v)\setminus X=G\setminus X\setminus v=H_{3}$.

Therefore we may assume that $m\neq 0$. Let $H'=G*u_{1}$. Then $H=H'*u_{2}*\cdots *u_{m}\setminus X$, $H_{1}'=H'\setminus v$, $H_{2}'=H'*v\setminus v$, and $H_{3}'=H'/v$. By the induction hypothesis, there is a permutation $\sigma_{1}:\{1,2,3\}\rightarrow\{1,2,3\}$ such that $H_{i}$ is a vertex-minor of $H'_{\sigma_{1}(i)}$ for each $i\in\{1,2,3\}$. By Lemma~\ref{lem:onetoone}, there is a permutation $\sigma_{2}:\{1,2,3\}\rightarrow\{1,2,3\}$ such that $H'_{j}$ is locally equivalent to $G_{\sigma_{2}(j)}$ for each $j\in\{1,2,3\}$. Let $\sigma=\sigma_{2}\circ\sigma_{1}$. Then $H_{i}$ is a vertex-minor of $G_{\sigma(i)}$ for each $i\in\{1,2,3\}$. 
\end{proof}

\paragraph{Cut-rank function and connectivity}
For a finite set $V$, a $V\times V$-matrix $A$, 
and subsets $X$ and $Y$ of $V$, let $A[X,Y]$ be the $X\times Y$-submatrix of $A$. 
For a graph $G$, let $A_{G}$ be the adjacency matrix of~$G$ over the binary field $\GF(2)$. The \emph{cut-rank} $\rho_{G}(X)$ of $X\subseteq V(G)$ is defined by 
\[\rho_G(X)=\rnk(A_{G}[X,V(G)-X]).\] 
It is obvious to check that $\rho_{G}(X)=\rho_{G}(V(G)-X)$. 

The following lemmas give some properties of the cut-rank function.
\begin{lemma}[see Oum~{\cite[Proposition 2.6]{Oum2005}}]
\label{lem:local}
If a graph $G'$ is locally equivalent to a graph $G$, then $\rho_{G}(X)=\rho_{G'}(X)$ for each $X\subseteq V(G)$.
\end{lemma}

\begin{lemma}[see Oum~{\cite[Corollary 4.2]{Oum2005}}]
\label{lem:subeq}
Let $G$ be a graph and let $X$, $Y$ be subsets of $V(G)$. Then
\[
\rho_{G}(X)+\rho_{G}(Y)\geq\rho_{G}(X\cap Y)+\rho_{G}(X\cup Y).
\]
\end{lemma}

\begin{lemma}[Oum~{\cite[Lemma 2.3]{Oum2020}}] 
\label{lem:subtool}
Let $G$ be a graph and $v$ be a vertex of $G$. Let $X$ and $Y$ be subsets of $V(G)-\{v\}$. Then the following hold:
\begin{enumerate}[label=\rm(S\arabic*)]
\item\label{item:s1} $\rho_{G\setminus v}(X)+\rho_{G}(Y\cup\{v\})\geq\rho_{G\setminus v}(X\cap Y)+\rho_{G}(X\cup Y\cup\{v\})$.
\item\label{item:s2} $\rho_{G\setminus v}(X)+\rho_{G}(Y)\geq\rho_{G}(X\cap Y)+\rho_{G\setminus v}(X\cup Y)$.
\end{enumerate}
\end{lemma}

\begin{lemma}
\label{lem:delrank}
Let $G$ be a graph and $v$ be a vertex of $G$. For a subset $X$ of $V(G)-\{v\}$, we have
\begin{enumerate}[label=\rm(\roman*)]
\item\label{item:2.4i} $\rho_{G\setminus v}(X)+1\geq\rho_{G}(X)\geq\rho_{G\setminus v}(X)$.
\item\label{item:2.4ii} $\rho_{G\setminus v}(X)+1\geq\rho_{G}(X\cup\{v\})\geq\rho_{G\setminus v}(X)$.
\end{enumerate}
\end{lemma}
\begin{proof}
Observe that removing a row or a column of a matrix decreases the rank by at most~$1$ and never increases the rank.
\end{proof}

Let $G$ be a graph and $S$, $T$ be disjoint subsets of $V(G)$. 
The \emph{connectivity between $S$ and $T$ in $G$}, denoted by $\kappa_{G}(S,T)$, is defined by $\min_{S\subseteq X\subseteq V(G)-T}\rho_{G}(X)$. 

\begin{lemma}
\label{lem:kmonotone}
Let $H$ be a vertex-minor of a graph $G$ and $S$ and $T$ be disjoint subsets of $V(H)$. Then $\kappa_{H}(S,T)\leq\kappa_{G}(S,T)$.
\end{lemma}
\begin{proof}
The conclusion follows from Lemma~\ref{lem:local} and (i) of Lemma~\ref{lem:delrank}.
\end{proof}

\begin{lemma}[Oum and Seymour~{\cite[Lemma 1]{Oum2007}}]
\label{lem:subconn}
Let $G$ be a graph and $X_{1}$, $X_{2}$, $Y_{1}$, and $Y_{2}$ be subsets of $V(G)$ such that $X_{1}\cap X_{2}=Y_{1}\cap Y_{2}=\emptyset$. Then, we have
\[
\kappa_{G}(X_{1},X_{2})+\kappa_{G}(Y_{1},Y_{2})\geq\kappa_{G}(X_{1}\cap Y_{1},X_{2}\cup Y_{2})+\kappa_{G}(X_{1}\cup Y_{1},X_{2}\cap Y_{2}).
\]
\end{lemma}

The following corollaries are easy consequences of Theorem~\ref{thm:oum_linking}.

\begin{corollary}
\label{cor:one}
Let $G$ be a graph and $Q$, $R$, $S$, and $T$ be subsets of $V(G)$ such that $Q\cap R=S\cap T=\emptyset$. Let $F=V(G)-(Q\cup R\cup S\cup T)$, $k=\kappa_{G}(Q,R)$, and $\ell=\kappa_{G}(S,T)$. For every vertex $v$ of~$F$, at least one of the following holds:
\begin{enumerate}[label=\rm(\roman*)]
\item $\kappa_{G\setminus v}(Q,R)=k$ and $\kappa_{G\setminus v}(S,T)=\ell$.
\item $\kappa_{G*v\setminus v}(Q,R)=k$ and $\kappa_{G*v\setminus v}(S,T)=\ell$.
\item $\kappa_{G\wedge uv\setminus v}(Q,R)=k$ and $\kappa_{G\wedge uv\setminus v}(S,T)=\ell$ for each neighbor $u$ of $v$.
\end{enumerate}
\end{corollary}
\begin{proof}
  By Theorem~\ref{thm:oum_linking}, at least two graphs $H_1$, $H_2$ among  $G\setminus v$, 
  $G*v\setminus v$, and $G/v$ have the property that $\kappa_{H_1}(Q,R)=\kappa_{H_2}(Q,R)=k$.
  Again by Theorem~\ref{thm:oum_linking}, 
  at least one graph $H$ of $H_1$ or $H_2$ satisfies the property that 
  $\kappa_H(S,T)=\ell$. 
\end{proof}
\begin{corollary}
\label{cor:One}
Let $G$ be a graph and $Q$, $R$, $S$, and $T$ be subsets of $V(G)$ such that $Q\cap R=S\cap T=\emptyset$. Let $F$ be a subset of $V(G)-(Q\cup R\cup S\cup T)$, $k=\kappa_{G}(Q,R)$, and $\ell=\kappa_{G}(S,T)$. Then there exists a vertex-minor $H$ of $G$ such that $V(H)=V(G)-F$, $\kappa_{H}(Q,R)=k$, and $\kappa_{H}(S,T)=\ell$.
\end{corollary}
\begin{proof}
We proceed by induction on $|F|$. We may assume that $|F|\geq 1$. Let $v$ be a vertex of $F$. By Corollary~\ref{cor:one}, there is a graph $G_{1}\in \{G\setminus v, G*v\setminus v, G/v\}$ such that $\kappa_{G_{1}}(Q,R)=k$ and $\kappa_{G_{1}}(S,T)=\ell$. By the induction hypothesis, there is a vertex-minor $H$ of $G_{1}$ such that $V(H)=V(G_{1})-(F-\{v\})=V(G)-F$, $\kappa_{H}(Q,R)=\kappa_{G_{1}}(Q,R)=k$, and $\kappa_{H}(S,T)=\kappa_{G_{1}}(S,T)=\ell$. Therefore, the conclusion follows since $H$ is a vertex-minor of $G$.
\end{proof}

The following lemma is the analog of {\cite[Lemma 4.7]{Geelen2007}}.
\begin{lemma}
\label{lem:base}
Let $G$ be a graph and $S$ and $T$ be disjoint subsets of $V(G)$. Then there exist $S_{1}\subseteq S$ and $T_{1}\subseteq T$ such that $|S_{1}|=|T_{1}|=\kappa_{G}(S_{1},T_{1})=\kappa_{G}(S,T)$.
\end{lemma}
\begin{proof}
By Lemma~\ref{lem:subconn}, there exists a matroid $M_{1}$ on $V(G)-T$ whose rank function is 
$\kappa_{G}(X,T)$ for each subset $X$ of $V(G)-T$. Let $S_{1}$ be a maximal independent set of $M_{1}$ contained in $S$. Then we have $|S_{1}|=\kappa_{G}(S_{1},T)=\kappa_{G}(S,T)$. By Lemma~\ref{lem:subconn}, there is a matroid $M_{2}$ on $V(G)-S_{1}$ whose rank function is $\kappa_{G}(X,S_{1})$ for every subset $X$ of $V(G)-S_{1}$. Let $T_{1}$ be a maximal independent set of $M_{2}$ contained in $T$. Then $|T_{1}|=\kappa_{G}(T_{1},S_{1})=\kappa_{G}(T,S_{1})$ and so we finish the proof.
\end{proof}

\section{Lemmas on the cut-rank function.}
\label{sec:pf}
In this section, we present simple lemmas on the cut-rank function.
A subset $X$ of $V(G)$ is an \emph{$(S,T)$-separating set of order $k$} in $G$ if $S\subseteq X\subseteq  V(G)-T$ and $\rho_{G}(X)=k$.

For a graph $G$ and disjoint subsets $S$, $T$ of $V(G)$, let $\tilde\sqcap_{G}[S,T]=\frac{1}{2}(\rho_{G}(S)+\rho_{G}(T)-\rho_{G}(S\cup T))$. 

\begin{lemma}
\label{lem:capcup}
Let $G$ be a graph and $S$ and $T$ be disjoint subsets of $V(G)$. 
If $A$ and $B$ are $(S,T)$-separating sets of order $k:=\kappa_{G}(S,T)$ in $G$, then both $A\cap B$ and $A\cup B$ are $(S,T)$-separating sets of order $k$ in $G$.
\end{lemma}
\begin{proof}
  Since both $A\cap B$ and $A\cup B$ are $(S,T)$-separating sets, $\rho_{G}(A\cap B)\ge k$ and $\rho_{G}(A\cup B)\ge k$.
By Lemma~\ref{lem:subeq}, 
\[
2k=\rho_{G}(A)+\rho_{G}(B)\geq\rho_{G}(A\cup B)+\rho_{G}(A\cap B)\ge 2k
\] 
and therefore $\rho_{G}(A\cup B)=\rho_{G}(A\cap B)=k$.
\end{proof}

\begin{lemma}
\label{lem:nonflex}
Let $G$ be a graph and $S$ and $T$ be disjoint subsets of $V(G)$ such that $\rho_G(S)=\kappa_G(S,T)$. Let $U$ be a subset of $S$. Let $v$ be a vertex in $V(G)-(S\cup T)$. If $\kappa_{G\setminus v}(U,T)<\kappa_{G}(U,T)$, then $\kappa_{G\setminus v}(S,T)<\kappa_{G}(S,T)$.
\end{lemma}
\begin{proof}
  Let $k=\rho_G(S)=\kappa_G(S,T)$.
  Suppose that $\kappa_{G\setminus v}(S,T)=k$. 
  Let $X$ be a $(U,T)$-separating set in $G\setminus v$. 
  By~\ref{item:s2} of Lemma~\ref{lem:subtool},
  \[
  \rho_{G\setminus v}(X)+\rho_G(S)
  \ge \rho_G(X\cap S)+\rho_{G\setminus v}(X\cup S)
  \] 
  and since $X\cup S$ is $(S,T)$-separating in $G\setminus v$, 
  we have $\rho_{G\setminus v}(X\cup S)\ge k=\rho_G(S)$. 
  Hence, we deduce that $\rho_{G\setminus v}(X)\ge 
  \rho_{G}(X\cap S)\ge \kappa_G(U,T)$.
  So $\kappa_{G\setminus v}(U,T)\ge\kappa_{G}(U,T)$, contradicting the assumption.
\end{proof}

\begin{lemma}
\label{lem:conn}
Let $G$ be a graph and $X_{2}$ and $Y$ be disjoint subsets of $V(G)$. Let $X_{1}$ be a subset of $X_{2}$. Then $\tilde\sqcap_{G}[X_{1},Y]\leq\tilde\sqcap_{G}[X_{2},Y]$.
\end{lemma}
\begin{proof}
Since $X_{1}\subseteq X_{2}$, by Lemma~\ref{lem:subeq}, we have 
\begin{align*}
\rho_{G}(X_{2})+\rho_{G}(X_{1}\cup Y)&\geq\rho_{G}(X_{2}\cup(X_{1}\cup Y))+\rho_{G}(X_{2}\cap(X_{1}\cup Y)) \\
&=\rho_{G}(X_{2}\cup Y)+\rho_{G}(X_{1}).
\end{align*}
Hence, $2\tilde\sqcap_{G}(X_{1},Y)=\rho_{G}(X_{1})+\rho_{G}(Y)-\rho_{G}(X_{1}\cup Y)\leq \rho_{G}(X_{2})+\rho_{G}(Y)-\rho_{G}(X_{2}\cup Y)=2\tilde\sqcap_{G}(X_{2},Y)$.
\end{proof}

\begin{lemma}
\label{lem:qset}
Let $G$ be a graph and $Q$ and $R$ be disjoint subsets of $V(G)$ such that $\rho_{G}(Q)=\kappa_{G}(Q,R)$. Let $v$ be a vertex of $V(G)-(Q\cup R)$ such that $\kappa_{G\setminus v}(Q,R)<\kappa_{G}(Q,R)$. Then the following hold:
\begin{enumerate}[label=\rm(Q\arabic*)]
\item\label{item:q1} $\rho_{G}(Q\cup\{v\})\geq\rho_{G}(Q)$.
\item\label{item:q2} If $\rho_{G\setminus v}(Q)=\rho_{G}(Q)$, then $\rho_{G}(Q\cup\{v\})=\rho_{G}(Q)+1$.
\end{enumerate}
\end{lemma}
\begin{proof}
\ref{item:q1} holds clearly since $\rho_{G}(Q)=\kappa_{G}(Q,R)$.

To prove~\ref{item:q2}, let $k=\kappa_{G}(Q,R)$. Since $\kappa_{G\setminus v}(Q,R)<k$, there is a subset $X$ of $V(G)$ such that $Q\subseteq X\subseteq V(G)-(R\cup\{v\})$ and $\rho_{G\setminus v}(X)\leq k-1$. Then $\rho_{G\setminus v}(X)<k\leq\rho_{G}(X\cup\{v\})$ because $Q\subseteq X\cup\{v\}\subseteq V(G)-R$ and by~\ref{item:s1} of Lemma~\ref{lem:subtool}, we have that 
\[
\rho_{G\setminus v}(X)+\rho_{G}(Q\cup\{v\})\geq\rho_{G\setminus v}(Q)+\rho_{G}(X\cup\{v\})>\rho_{G\setminus v}(Q)+\rho_{G\setminus v}(X).
\]
Hence, by Lemma~\ref{lem:delrank}, $\rho_{G}(Q\cup\{v\})=\rho_{G\setminus v}(Q)+1=\rho_{G}(Q)+1$.
\end{proof}

\section{Proof of Theorem~\ref{thm:main}}
\label{sec:mainpf}

For disjoint subsets $S$ and $T$ of vertices of a graph $G$, a vertex $v\in V(G)-(S\cup T)$ is \emph{$(S,T)$-flexible} if $\kappa_{G\setminus v}(S,T)=\kappa_{G*v\setminus v}(S,T)=\kappa_{G\wedge uv\setminus v}(S,T)=\kappa_{G}(S,T)$ for each $u\in N_{G}(v)$. Note that every isolated vertex is $(S,T)$-flexible. 

\begin{lemma}
\label{lem:maintain_flexible}
Let $S$, $T$ be disjoint sets of vertices of a graph $G$. If a vertex $v$ is $(S,T)$-flexible in $G$, 
then it is $(S,T)$-flexible in every graph locally equivalent to $G$.
\end{lemma}
\begin{proof}
  Let $G'$ be a graph locally equivalent to $G$.
Let $k=\kappa_{G}(S,T)$, $G_{1}=G\setminus v$, $G_{2}=G*v\setminus v$, and $G_{3}=G/v$. Since $v$ is $(S,T)$-flexible in $G$, we have $\kappa_{G_{1}}(S,T)=\kappa_{G_{2}}(S,T)=\kappa_{G_{3}}(S,T)=k$. Let $H_{1}=G'\setminus v$, $H_{2}=G'*v\setminus v$, and $H_{3}=G'/v$. Then by Lemma~\ref{lem:perm}, there is a permutation $\sigma:\{1,2,3\}\rightarrow\{1,2,3\}$ such that $H_{i}$ is locally equivalent to $G_{\sigma(i)}$ for each $i\in\{1,2,3\}$. Hence, by Lemma~\ref{lem:local}, we have $\kappa_{H_{i}}(S,T)=\kappa_{G_{\sigma(i)}}(S,T)=k$ for each $i\in\{1,2,3\}$. Therefore, $v$ is $(S,T)$-flexible in $G'$.
\end{proof}

The following lemma finds a nested set of $(S,T)$-separating sets of order $\kappa_{G}(S,T)$ for disjoint sets $S$ and $T$ of vertices of a graph $G$.

\begin{lemma}
\label{lem:seqsep}
Let $G$ be a graph and $S$ and $T$ be disjoint subsets of $V(G)$. Let $k=\kappa_{G}(S,T)$ and $F\subseteq V(G)-(S\cup T)$ be a set of $n$ vertices which are not $(S,T)$-flexible. Then there exist an ordering $f_{1},\ldots,f_{n}$ of vertices in $F$ and a sequence $A_{1},\ldots,A_{n}$ of $(S,T)$-separating sets of order $k$ in $G$ such that the following hold:
\begin{enumerate}[label=\rm(\roman*)]
\item\label{item:a1} $A_{i}\subseteq A_{i+1}$ for each $1\leq i\leq n-1$.
\item\label{item:a2} $A_{i}\cap F=\{f_{1},\ldots,f_{i}\}$ for each $1\leq i\leq n$.
\end{enumerate}
\end{lemma}
\begin{proof}
We prove by induction on $n=|F|$. We may assume that $n\geq 1$. We first claim that for every $v\in F$, there exists an $(S,T)$-separating set of order $k$ in $G$ containing $v$. 
Since $v$ is not $(S,T)$-flexible in $G$, there exists a graph $G'\in\{G\setminus v, G*v\setminus v, G/v\}$ such that $\kappa_{G'}(S,T)<\kappa_{G}(S,T)$. So there is a subset $A$ of $V(G)-\{v\}$ such that $S\subseteq A\subseteq V(G)-(T\cup\{v\})$ and $\rho_{G'}(A)\leq k-1$. There exists a graph $H$ locally equivalent to $G$ such that $H\setminus v=G'$. Therefore, since $S\subseteq A\cup\{v\}\subseteq V(G)-T$, by Lemmas~\ref{lem:local} and \ref{lem:delrank}, we have $k\leq\rho_{G}(A\cup\{v\})=\rho_{H}(A\cup\{v\})\leq\rho_{H\setminus v}(A)+1=\rho_{G'}(A)+1\leq k$ and so $\rho_{G}(A\cup\{v\})=k$. Now it follows that $A\cup\{v\}$ is an $(S,T)$-separating set of order $k$ in $G$ containing $v$.

For each $u\in F$, let $A_{u}$ be an $(S,T)$-separating set of order $k$ in $G$ containing $u$ such that $|A_{u}|$ is minimum. Let $x$ be a vertex of $F$ such that $|A_{x}|\leq |A_{u}|$ for each $u\in F$. 

Now we claim that $A_{x}\cap F=\{x\}$.
Suppose that there exists an element $y\in (A_{x}-\{x\})\cap F$. 
Then, by Lemma~\ref{lem:capcup}, both $A_{x}\cap A_{y}$ and $A_{x}\cup A_{y}$ are $(S,T)$-separating sets of order $k$ in $G$.
Hence, $A_{y}\subseteq A_{x}$ by the choice of $A_{y}$. Then we have $A_{x}=A_{y}$ because $|A_{x}|\leq|A_{u}|$ for every $u\in F$. Since $y$ is not $(S,T)$-flexible, there exists a graph $G''\in\{G\setminus y, G*y\setminus y, G/y\}$ such that 
$\kappa_{G''}(S,T)<\kappa_{G}(S,T)$. By Lemma~\ref{lem:local}, we may assume that $G''=G\setminus y$. Then there exists $S\subseteq X\subseteq V(G)-(T\cup\{y\})$ such that $\rho_{G\setminus y}(X)=k-1$. By Lemma~\ref{lem:delrank}, $\rho_{G}(X)=k$ and $\rho_{G}(X\cup\{y\})=k$.
So $X\cup\{y\}$ is an $(S,T)$-separating set of order $k$ in $G$ containing $y$. By Lemma~\ref{lem:capcup}, $A_{y}\cap(X\cup\{y\})$ is an $(S,T)$-separating set of order $k$ in $G$. Therefore, by the choice of $A_{y}$, we have $A_{y}\subseteq X\cup\{y\}$ and so $A_{y}-\{y\}\subseteq X$. By applying \ref{item:s1} of Lemma~\ref{lem:subtool}, 
\begin{align*}
2k-1&=\rho_{G\setminus y}(X)+\rho_{G}(A_{y})=\rho_{G\setminus y}(X)+\rho_{G}((A_{y}-\{y\})\cup\{y\}) \\
&\geq\rho_{G\setminus y}(X\cap(A_{y}-\{y\}))+\rho_{G}(X\cup(A_{y}-\{y\})\cup\{y\}) \\
&=\rho_{G\setminus y}(A_{y}-\{y\})+\rho_{G}(X\cup\{y\}) .
\end{align*}
Since $\rho_{G}(X\cup\{y\})=k$, we know that $\rho_{G\setminus y}(A_{y}-\{y\})\leq k-1$ and so $\rho_{G}(A_{y}-\{y\})\leq k$ by Lemma~\ref{lem:delrank}. Recall that $S\subseteq A_{y}-\{y\}\subseteq V(G)-T$ and $k=\kappa_{G}(S,T)$. Therefore, $\rho_{G}(A_{y}-\{y\})=k$. Since $A_{x}=A_{y}$, this is a contradiction to the minimality of $A_{x}$. Thus $A_{x}\cap F=\{x\}$.

Let $f_{1}=x$ and $A_{1}=A_{x}$. Then $k=\kappa_{G}(S,T)\leq\kappa_{G}(A_{1},T)\leq\rho_{G}(A_{1})=k$ and therefore we have that $\kappa_{G}(A_{1},T)=k$. By Lemmas~\ref{lem:local} and \ref{lem:nonflex}, no vertex of $F-\{f_{1}\}$ is $(A_{1},T)$-flexible. Hence, by the induction hypothesis, there exist an ordering $f_{2},\ldots,f_{n}$ of elements of $F-\{f_{1}\}$ and a sequence $A_{2},\ldots,A_{n}$ of $(A_{1},T)$-separating sets of order $k$ in $G$ such that \ref{item:a1} and \ref{item:a2} hold.

So we finish the proof with the fact that $A_{2},\ldots,A_{n}$ are also $(S,T)$-separating sets of order $k$ in~$G$.
\end{proof}

Our proof of Theorem~\ref{thm:main} consists of two parts. In the first part, we will assume that 
$S$ and $T$ are small and prove the theorem. In the second part, we will show how to reduce the size of $S$ and $T$. The following lemma will be used at the key step in the first part.

\begin{lemma}
\label{lem:nesting}
Let $G$ be a graph and $Q$, $R$, $S$, and $T$ be subsets of $V(G)$ such that $Q\cap R=S\cap T=\emptyset$ and $S\cup T\subseteq Q\cup R$. Let $F=V(G)-(Q\cup R)\neq\emptyset$ and $k=\kappa_{G}(Q,R)$ and $\ell=\kappa_{G}(S,T)$. If $\rho_{G}(Q)=\rho_{G}(R)=k$ and no vertex of $F$ is $(Q,R)$-flexible or $(S,T)$-flexible, then (1) or (2) holds:
\begin{enumerate}[label=\rm(\arabic*)]
\item\label{item:ext1} There exists a vertex $v$ of $F$ such that at least two of the following hold:
\begin{enumerate}[label=\rm(\roman*)]
\item $\kappa_{G\setminus v}(Q,R)=k$ and $\kappa_{G\setminus v}(S,T)=\ell$.
\item $\kappa_{G*v\setminus v}(Q,R)=k$ and $\kappa_{G*v\setminus v}(S,T)=\ell$.
\item $\kappa_{G\wedge uv\setminus v}(Q,R)=k$ and $\kappa_{G\wedge uv\setminus v}(S,T)=\ell$ for each $u\in N_{G}(v)$.
\end{enumerate}
\item\label{item:ext2} There exist disjoint subsets $Q'$ and $R'$ of $V(G)$ such that the following hold:
\begin{enumerate}[label=\rm(\roman*)]
\item\label{item:1} $Q\subseteq Q'$, $R\subseteq R'$ and $\rho_{G}(Q')=\rho_{G}(R')=k$.
\item\label{item:2} $\tilde\sqcap_{G}[Q',R']\geq\tilde\sqcap_{G}[Q,R]+\frac{1}{2}$.
\item\label{item:3} $\left|V(G)-(Q'\cup R')\right|\geq\lfloor\frac{1}{2}|F|\rfloor$.
\end{enumerate}
\end{enumerate}
\end{lemma}
\begin{proof}
Assume that \ref{item:ext1} does not hold.
Let $n=|F|$. Since no vertex of $F$ is $(Q,R)$-flexible, by Lemma~\ref{lem:seqsep}, there exists an ordering $f_{1},\ldots,f_{n}$ of vertices of $F$ such that $Q\cup\{f_{1},\ldots,f_{i}\}$ is a $(Q,R)$-seperating set of order $k$ in $G$ for each $i\in\{1,\ldots,n\}$. Let $A_{i}= Q\cup\{f_{1},\ldots,f_{i}\}$ for each $1\leq i\leq n$.

No vertex of $F$ is $(S,T)$-flexible and so, by Lemma~\ref{lem:seqsep}, there exist a vertex $g$ in $F$ and an $(S,T)$-seperating set $C$ of order $\ell$ in $G$ such that $C-(Q\cup R)=\{g\}$.

By Theorem~\ref{thm:oum_linking}, there are graphs $G_{1}', G_{2}'\in\{G\setminus g, G*g\setminus g, G/g\}$ such that $\kappa_{G_{i}'}(S,T)=\kappa_{G}(S,T)$ for $i\in\{1,2\}$. Since \ref{item:ext1} does not hold, there exists $G'\in\{G_{1}',G_{2}'\}$ such that $\kappa_{G'}(Q,R)<\kappa_{G}(Q,R)$. Then by Lemma~\ref{lem:local}, we may assume that $G'=G\setminus g$. 

Since $\kappa_{G\setminus g}(S,T)=\ell$ and $S\subseteq C-\{g\}\subseteq V(G\setminus g)-T$, we have $\ell\leq\rho_{G\setminus g}(C-\{g\})\leq\rho_{G}(C)=\ell$ and therefore 
$\rho_{G\setminus g}(C-\{g\})=\rho_{G}(C)$. Since $C-\{g\}\subseteq Q\cup R$, by \ref{item:s1} of Lemma~\ref{lem:subtool},
\begin{align*}
\rho_{G\setminus g}(Q\cup R)+\rho_{G}(C)&\geq\rho_{G\setminus g}((Q\cup R)\cap C)+\rho_{G}((Q\cup R)\cup C) \\
&=\rho_{G\setminus g}(C-\{g\})+\rho_{G}(Q\cup R\cup\{g\}).
\end{align*}
Hence $\rho_{G}(Q\cup R\cup\{g\})\leq\rho_{G\setminus g}(Q\cup R)$ because $\rho_{G\setminus g}(C-\{g\})=\rho_{G}(C)$. By Lemma~\ref{lem:delrank}, $\rho_{G\setminus g}(Q\cup R)\leq\rho_{G}(Q\cup R\cup\{g\})$ and therefore 
$\rho_{G\setminus g}(Q\cup R)=\rho_{G}(Q\cup R\cup\{g\})$.

Now we claim that $\tilde\sqcap_{G}(Q\cup\{g\},R)\geq\tilde\sqcap_{G}(Q,R)+\frac{1}{2}$. Observe that it is equivalent to show that
\[
\rho_{G}(Q\cup\{g\})+\rho_{G}(R)-\rho_{G}(Q\cup R\cup\{g\})\geq\rho_{G}(Q)+\rho_{G}(R)-\rho_{G}(Q\cup R)+1.
\]
We have $\rho_{G}(Q\cup R)\geq\rho_{G\setminus g}(Q\cup R)=\rho_{G}(Q\cup R\cup\{g\})$ 
and, by \ref{item:q1} of Lemma~\ref{lem:qset}, $\rho_{G}(Q\cup\{g\})\geq\rho_{G}(Q)$.
Therefore, it is enough to prove that $\rho_{G}(Q\cup R)\geq\rho_{G}(Q\cup R\cup\{g\})+1$ or $\rho_{G}(Q\cup\{g\})\geq\rho_{G}(Q)+1$.
Suppose that $\rho_{G}(Q\cup R)=\rho_{G}(Q\cup R\cup\{g\})=\rho_{G\setminus g}(Q\cup R)$.
Then, by \ref{item:s2} of Lemma~\ref{lem:subtool}, we have
\[
\rho_{G\setminus g}(Q)+\rho_{G}(Q\cup R)\geq\rho_{G\setminus g}(Q\cup R)+\rho_{G}(Q).
\]
So $\rho_{G\setminus g}(Q)\geq\rho_{G}(Q)$ and we have $\rho_{G\setminus g}(Q)=\rho_{G}(Q)$ by Lemma~\ref{lem:delrank}. Then by \ref{item:q2} of Lemma~\ref{lem:qset}, $\rho_{G}(Q\cup\{g\})=\rho_{G}(Q)+1$, proving the claim.

Similarly, we have $\tilde\sqcap_{G}(Q,R\cup\{g\})\geq\tilde\sqcap_{G}(Q,R)+\frac{1}{2}$. Let $i$ be an integer such that $f_{i}=g$ and let
\[
(Q',R')=
\begin{cases}
(A_{i},R) & \text{if $i\leq \lfloor\frac{n}{2}\rfloor$,} \\
(Q,V(G)-A_{i-1}) & \text{otherwise.}
\end{cases}
\]
Then by Lemma~\ref{lem:conn}, \[\tilde\sqcap_{G}(Q',R')\geq\min\left(\tilde\sqcap_{G}(Q\cup\{g\},R),\tilde\sqcap_{G}(Q,R\cup\{g\})\right)\geq\tilde\sqcap_{G}(Q,R)+\frac{1}{2}.\] So \ref{item:2} holds and \ref{item:1} and \ref{item:3} hold by the construction.
\end{proof}

Now we are ready to prove Theorem~\ref{thm:main} when $S$ and $T$ are small.

\begin{proposition}
\label{prop:ST_small}
Let $G$ be a graph and $Q$, $R$, $S$, and $T$ be subsets of $V(G)$ such that $Q\cap R=S\cap T=\emptyset$ and $F=V(G)-(Q\cup R\cup S\cup T)$. Let $k=\kappa_{G}(Q,R)$ and $\ell=\kappa_{G}(S,T)$. If $|S|=|T|=\ell$ and $|F|\geq (2\ell+1)2^{2k}$, then there is a vertex $v\in F$ such that at least two of the following hold:
\begin{enumerate}[label=\rm(\arabic*)]
\item $\kappa_{G\setminus v}(Q,R)=k$ and $\kappa_{G\setminus v}(S,T)=\ell$.
\item $\kappa_{G*v\setminus v}(Q,R)=k$ and $\kappa_{G*v\setminus v}(S,T)=\ell$.
\item $\kappa_{G\wedge uv\setminus v}(Q,R)=k$ and $\kappa_{G\wedge uv\setminus v}(S,T)=\ell$ for every neighbor $u$ of $v$.
\end{enumerate}
\end{proposition}
\begin{proof}
If $F$ has a vertex which is $(S,T)$-flexible or $(Q,R)$-flexible, then our conclusion follows by Theorem~\ref{thm:oum_linking}. So we can assume that no vertex of $F$ is $(S,T)$-flexible or $(Q,R)$-flexible. Let $n=|F|$.

By Lemma~\ref{lem:seqsep}, there exist an ordering $f_{1},\ldots,f_{n}$ of vertices of $F$ and a sequence $A_{1},\ldots,A_{n}$ of $(Q,R)$-seperating sets of order $k$ in $G$ satisfying the following:
\begin{itemize}
\item $A_{i}\subseteq A_{i+1}$ for each $1\leq i\leq n-1$.
\item $A_{i}\cap F=\{f_{1},\ldots,f_{i}\}$ for each $1\leq i\leq n$.
\end{itemize}
For each $1\leq i\leq n$, let $B_{i}=V(G)-A_{i}$. Let $q=2^{2k}$ and $A_{0}=Q$. For $1\leq i\leq 2\ell+1$, let $X_{i}= A_{iq}-A_{(i-1)q}$. Since $|S|=|T|=\ell$, there exists $1\leq m\leq 2\ell+1$ such that $X_{m}\cap(S\cup T)=\emptyset$. 
Let $j=(m-1)q$. Then we have $Q\cup R\cup S\cup T\subseteq A_{j}\cup B_{j+q}$. 

Assume that our conclusion fails and so every vertex of $F$ satisfies at most one of (1), (2), and (3).
We claim that, for each $1\leq i\leq 2k+2$, there exist disjoint subsets $Q_{i}$ and $R_{i}$ of $V(G)$ satisfying the following.
\begin{enumerate}[label=\rm(\roman*)]
\item\label{item:_1} $Q \subseteq Q_{i}$, $R \subseteq R_{i}$, and $\rho_{G}(Q_{i})=\rho_{G}(R_{i})=k$.
\item\label{item:_2} $\tilde\sqcap_{G}[Q_{i},R_{i}]\geq\frac{i-1}{2}$.
\item\label{item:_3} $\left|V(G)-(Q_{i}\cup R_{i})\right|\geq \lfloor2^{2k+1-i}\rfloor$.
\end{enumerate}
We proceed by the induction on $i$. Let $Q_{1}=A_{j}$, $R_{1}=B_{j+q}$, and $F_{1}=V(G)-(Q_{1}\cup R_{1})$. Then $|F_{1}|=2^{2k}$ and so $(Q_{1},R_{1})$ satisfies the claim. Therefore we may assume that $i\geq 2$. By the induction hypothesis, there exist disjoint subsets $Q_{i-1}$ and $R_{i-1}$ of $V(G)$ satisfying~\ref{item:_1}, \ref{item:_2}, and~\ref{item:_3} for $i-1$.
By Lemmas~\ref{lem:local} and \ref{lem:nonflex}, no vertex of $V(G)-(Q_{i-1}\cup R_{i-1})$ is $(Q_{i-1},R_{i-1})$-flexible. If there is a vertex $v$ of $V(G)-(Q_{i-1}\cup R_{i-1})$ satisfying~\ref{item:ext1} of Lemma~\ref{lem:nesting} for two pairs $(Q_{i-1}, R_{i-1})$ and $(S, T)$, 
then by Lemmas~\ref{lem:local} and~\ref{lem:nonflex}, $v$ satisfies at least two of (1), (2), and (3), contradicting our assumption. So we may assume that $V(G)-(Q_{i-1}\cup R_{i-1})$ has no such vertex.
Hence, by Lemma~\ref{lem:nesting}, there exist disjoint subsets $Q_{i}$ and $R_{i}$ of $V(G)$ such that the following hold:
\begin{enumerate}[label=\rm(\alph*)]
\item $Q_{i-1}\subseteq Q_{i}$, $R_{i-1}\subseteq R_{i}$ and $\rho_{G}(Q_{i})=\rho_{G}(R_{i})=k$.
\item $\tilde\sqcap_{G}[Q_{i},R_{i}]\geq\tilde\sqcap_{G}[Q_{i-1},R_{i-1}]+\frac{1}{2}\geq\frac{i-2}{2}+\frac{1}{2}=\frac{i-1}{2}$.
\item $\left|V(G)-(Q_{i}\cup R_{i})\right|\geq\lfloor\frac{1}{2}\left|V(G)-(Q_{i-1}\cup R_{i-1})\right|\rfloor\geq\lfloor\frac{1}{2}\cdot 2^{2k+2-i}\rfloor=\lfloor 2^{2k+1-i}\rfloor$.
\end{enumerate}
This proves our claim.
Then by (ii) and Lemma~\ref{lem:conn}, $k+\frac{1}{2}\leq\tilde\sqcap_{G}(Q_{2k+2},R_{2k+2})\leq\tilde\sqcap_{G}(Q_{2k+2},V(G)-Q_{2k+2})=\rho_{G}(Q_{2k+2})=k$, which is a contradiction. Therefore our conclusion holds.
\end{proof}

Now we are ready to complete the proof of Theorem~\ref{thm:main}.

\begin{proof}[Proof of Theorem~\ref{thm:main}]
By Lemma~\ref{lem:base}, there exist $S_{1}\subseteq S$ and $T_{1}\subseteq T$ such that $|S_{1}|=|T_{1}|=\kappa_{G}(S_{1},T_{1})=\kappa_{G}(S,T)$. Let $X=(S\cup T)-(Q\cup R\cup S_{1}\cup T_{1})$. 
By Corollary~\ref{cor:One}, there is a vertex-minor $H$ of $G$ such that $V(H)=V(G)-X$, $\kappa_{H}(Q,R)=k$, and $\kappa_{H}(S_{1},T_{1})=\ell$. 

For a vertex $v$ of $V(H)-(Q\cup R\cup S_{1}\cup T_{1})$, let $H_{1}^v=H\setminus v$, $H_{2}^v=H*v\setminus v$, and $H_{3}^v=H/v$ and let $G_{1}^v=G\setminus v$, $G_{2}^v=G*v\setminus v$, and $G_{3}^v=G/v$. Then by Lemma~\ref{lem:perm}, there exists a permutation $\sigma_v:\{1,2,3\}\rightarrow\{1,2,3\}$ such that $H_{i}^v$ is a vertex-minor of $G_{\sigma(i)}^{v}$ for each $i\in\{1,2,3\}$.
By Lemma~\ref{lem:kmonotone}, $\kappa_{H_{i}^v}(S_{1},T_{1})\leq\kappa_{G_{\sigma(i)}^{v}}(S_{1},T_{1})\leq\kappa_{G_{\sigma(i)}^{v}}(S,T)\leq\kappa_{G}(S,T)=\ell$ and $\kappa_{H_{i}^v}(Q,R)\leq\kappa_{G_{\sigma(i)}^{v}}(Q,R)\leq\kappa_{G}(Q,R)=k$ for each $i\in\{1,2,3\}$.

Since $|V(H)-(Q\cup R\cup S_{1}\cup T_{1})|=|F|\geq(2\ell+1)2^{2k}$, by Proposition~\ref{prop:ST_small}, there exist a vertex~$v$ of $V(H)-(Q\cup R\cup S_{1}\cup T_{1})=F$ and $i,j\in\{1,2,3\}$ such that $i\neq j$ and
$\kappa_{H_{i}^v}(Q,R)=\kappa_{H_{j}^v}(Q,R)=k$ and 
$\kappa_{H_{i}^v}(S_{1},T_{1})=\kappa_{H_{j}^v}(S_{1},T_{1})=\ell$. Therefore,  
$\kappa_{G_{\sigma(i)}^{v}}(S,T)=\kappa_{G_{\sigma(j)}^{v}}(S,T)=\ell$ and
$\kappa_{G_{\sigma(i)}^{v}}(Q,R)=\kappa_{G_{\sigma(j)}^{v}}(Q,R)=k$. 
\end{proof}

\paragraph{Acknowledgements.}
The authors would like to thank the anonymous reviewers for 
their careful reviews and suggestions.

\providecommand{\bysame}{\leavevmode\hbox to3em{\hrulefill}\thinspace}
\providecommand{\MR}{\relax\ifhmode\unskip\space\fi MR }
\providecommand{\MRhref}[2]{%
  \href{http://www.ams.org/mathscinet-getitem?mr=#1}{#2}
}
\providecommand{\href}[2]{#2}

\end{document}